\documentclass[11pt]{amsart}

\usepackage{amsmath, amsfonts, amssymb, amscd, amstext, amsthm}
\usepackage{latexsym}
\usepackage{mathrsfs}
\usepackage{mathscinet}
\usepackage{color}
\usepackage[hidelinks, pagebackref]{hyperref}
\usepackage{url}
\usepackage{setspace}
\usepackage{marginnote}
\long\def\@savemarbox#1#2{\global\setbox#1\vtop{\hsize\marginparwidth 
  \@parboxrestore\tiny\raggedright #2}}
\renewcommand*{\backref}[1]{}
\renewcommand*{\backrefalt}[4]{
  \ifcase #1
  [No citations.]
  \or [#2]
  \else [#2]
  \fi }

\AtBeginDocument{%
   \def\MR#1{}
}

\makeatletter
\@namedef{subjclassname@2020}{\textup{2020} Mathematics Subject Classification}
\makeatother

\newtheorem{theorem}{Theorem}[section]
\newtheorem{lemma}[theorem]{Lemma}
\newtheorem{proposition}[theorem]{Proposition}

\theoremstyle{definition}
\newtheorem{definition}[theorem]{Definition}

\theoremstyle{remark}
\newtheorem{remark}[theorem]{Remark}

\numberwithin{equation}{section}

\newcommand{\ts}{\hspace{.11111em}}
\newcommand{\tts}{\hspace{.05555em}}

\DeclareMathOperator{\G}{\operatorname{\mathcal{G}}}  

\DeclareMathOperator{\vol}{\operatorname{\mathsf{Vol}}\tts}
\DeclareMathOperator{\inj}{\operatorname{\mathsf{Inj}}\tts} 
\DeclareMathOperator{\emb}{\operatorname{\mathsf{Emb}}\tts} 
\DeclareMathOperator{\hyp}{\operatorname{\mathsf{hyp} }\tts}  

\newcommand{\RP}{\mathbb{R}\mathbb{P}}

\DeclareMathOperator{\sys}{\operatorname{\mathsf{Sys}}\tts}

\DeclareMathOperator{\SR}{\operatorname{\mathsf{SR}}\tts}

\begin{document} 

\baselineskip.525cm

\title[Triangulation Complexity and Systolic Volume]{Triangulation complexity and systolic volume of hyperbolic manifolds}

\author[L.~Chen]{L. Chen}

\address{
 \hspace*{0.055in}School of Mathematics and Statistics, Lanzhou University \newline
 \hspace*{0.175in} Lanzhou 730000, P.R. China 
}

\email{\hspace*{0.025in} chenzhmath@gmail.com}



\date{\today}

\begin{abstract} 
Let $M$ be a closed $n$-manifold with nonzero simplicial volume. A central result in systolic geometry from Gromov is that systolic volume of $M$ is related to its simplicial volume. In this short note, we show that systolic volume of hyperbolic manifolds is related to triangulation complexity. The proof is based on J{\o}rgensen and Thurston's theorem of hyperbolic volume.
\end{abstract}

\maketitle

\section{Introduction}

In systolic geometry, the systole of a Riemannian manifold is defined to be the shortest length of a noncontractible loop. Gromov's systolic inequality implies that systole of many manifolds (including hyperbolic manifolds) is bounded from above by Riemannian volume. The optimal constant in systolic inequality is usually called systolic volume in literature. Let $M$ be a closed $n$-dimensional manifold with nonzero simplicial volume. Gromov~\cite[Section 6.4.D]{gromov_filling_riemannian_manifolds_1983} proved that systolic volume is related to the topological complexity of $M$. More explicitly, Gromov proved that systolic volume relates to simplicial volume and other topological invariants. In this paper, we verify that systolic volume of hyperbolic manifolds is related to triangulation complexity. Triangulation complexity of a closed manifold is defined to be the minimum number of simplices in a triangulation. Hence the triangulation complexity naturally represents how complicated a manifold is. Our result is a supplement of Gromov's work.

Let $M$ be a closed $3$-manifold. If $M$ is irreducible and not homeomorphic to $S^3, \RP^3$ or $L(3, 1)$, the triangulation complexity coincides with the complexity defined by Matveev \cite{matveev_complexity}. In the following, we use $c(M)$ to denote the triangulation complexity of a closed $3$-manifold $M$. We refer to \cite{jaco_rubinstein_tillmann_complexity_13, jaco_rubinstein_tillmann_complexity_20, lackenby_purcell_complexity_21} for recent developments of triangulation complexity of 3-manifolds. The systolic volume of $M$, denoted by $\SR(M)$, is defined to be
\begin{equation*}
 \inf_{\G} \frac{\vol_{\G}(M)}{\sys \pi_1(M, \G)^n} , 
\end{equation*}
where the infimum is taken over all Riemannian metrics $\G$ on $M$. The systolic volume $\SR(M)$ is positive if $M$ is a closed essential $3$-manifold, see \cite[Theorem 0.1.A.]{gromov_filling_riemannian_manifolds_1983}.

Let $s_0$ be any sufficiently large positive number. 
\begin{theorem} \label{thm_main}
 Suppose that $M$ is a closed hyperbolic $3$-manifold with $\SR(M) \leqslant s_0$. There exists a positive constant $C(s_0)$ only depending on $s_0$, such that 
 \begin{equation*} 
  \SR(M) \geqslant C(s_0) \ts \frac{c(M)}{\log^3{c(M)}} .
 \end{equation*}
\end{theorem}

Simplicial volume $\| M \|$ of an oriented closed connected $n$-dimensional manifold $M$ is a homotopy invariant defined by Gromov~\cite{gromov_bounded_cohomology_1982}. The definition is
\begin{equation*}
 \| M \| = \left\{ \left. \sum_{i = 1}^k |r_i| \right| \sum_{i=1}^k r_i \sigma_i \, \text{represents the fundamental class} \, [M]  \right\},
\end{equation*}
where the infimum is taken over all real singular chains representing the fundamental class of $M$. Moreover, theorem of Gromov and Thurston implies that simplicial volume of hyperbolic $n$-manifolds $M$ ($n \geqslant 2$) satisfies
\begin{equation} \label{sv_hyp}
 \nu_n \ts \| M \| = \vol_{\hyp}(M) ,
\end{equation}
where $\nu_n$ is a positive constant only depending on $n$. Gromov indicated that systolic volume is related to simplicial volume. 
\begin{theorem}[Gromov, see {\cite[Section 6.4.D.]{gromov_filling_riemannian_manifolds_1983}} and {\cite[Section 3.C.3.]{gromov_systoles_notes_1996}}] \label{thm_Gromov_SR}
 Let $M$ be a closed $n$-dimensional manifold with non-zero simplicial volume. Then the systolic volume $\SR(M)$ of $M$ satisfies
\begin{equation} \label{Gromov}
 \| M \| \leqslant C_n \SR(M) \ts \log^n{ \left( C_n^{\prime} \ts \SR(M) \right)} ,
\end{equation}
  where $C_n$ and $C_n^{\prime}$ are two positive constants only depending on $n$. 
\end{theorem}
\begin{remark}
 In \cite[Section 6.4.D.]{gromov_filling_riemannian_manifolds_1983}, there is a typo of missing the exponent $n$ in the logarithm part of (\ref{Gromov}). In literature the estimate (\ref{Gromov}) in Theorem~\ref{thm_Gromov_SR} is also written to be
 \begin{equation*}
  \SR(M) \geqslant C_n \frac{\| M \|}{\log^n{\| M \|}},
 \end{equation*}
 where $C_n$ is a positive constant only depending on $n$, see \cite[Section 4.3.]{balacheff_karam_macroscopic_simplicial_2019}. 
\end{remark} 
Theorem \ref{thm_Gromov_SR} builds a bridge between systolic geometry and hyperbolic geometry. We refer to \cite{guth_metaphors_systolic_10} for more explanations about this interplay. Let $M$ be a hyperbolic manifold with volume $V$. Gromov and Thurston's theorem (see (\ref{sv_hyp})) yields that any triangulation of $M$ has at least $C \cdot V$ number of simplices for some constant $C$. Therefore, Gromov's theorem (Theorem~\ref{thm_Gromov_SR}) implies relation between topological complexity and systolic volume. More explicitly, if a hyperbolic manifold is topologically complicated, the systolic volume also should be sufficiently large. Our results in this paper can be seen as an attempt to interpret this relation with more details.

The main technique used in the proof of Theorem~\ref{thm_main} is the connection between triangulation and volume for hyperbolic manifolds. J{\o}rgensen and Thurston proved that any complete hyperbolic $3$-manifold admits a triangulation with the number of tetrahedra bounded from above by its volume. A detailed proof of this theorem is provided by Kobayashi and Rieck in \cite{kobayashi_rieck_hyperbolic_volume_2011}.

The triangulation complexity of manifolds in higher dimensions is studied in \cite{francaviglia_frigerio_martelli_stable_complexity}. Let $M$ be a closed manifold of dimension $n$. The triangulation complexity of $M$, denoted $\sigma(M)$, is defined to be the minimum number of $n$-simplices in any triangulation of $M$. When $n = 3$, $\sigma(M)$ coincides with $c(M)$.
\begin{theorem} \label{thm_02}
 Let $M$ be a closed hyperbolic manifold of dimension $n$ with $n \geqslant 4$. The triangulation complexity $\sigma(M)$ and systolic volume of $M$ is related by
  \begin{equation*}
   \SR(M) \geqslant D_n \frac{\sigma(M)}{\log^n{\sigma(M)}} ,
  \end{equation*}
  where $D_n$ is a positive constant only depending on $n$.  
 \end{theorem}
\begin{remark}
 It is proved in \cite{francaviglia_frigerio_martelli_stable_complexity} that under assumptions of Theorem \ref{thm_02}, $\| M \| < \sigma(M)$ holds. Hence above theorem is indeed a generalization of Gromov's theorem (Theorem \ref{thm_Gromov_SR}).  
\end{remark}

The embolic volume defined in terms of injectivity radius is another geometric quantitiy representing the topological complexity of manifolds. We refer to \cite[Section 11.2.3.]{berger_riemannian_geometry} for a general description.
\begin{definition} \label{def_emb}
 The embolic volume of $M$, denoted $\emb(M)$, is defined to be
 \begin{equation*}
  \inf_{\G} \frac{\vol_{\G}(M)}{\inj(M, \G)^n},
 \end{equation*}
 where the infimum is taken over all Riemannian metrics $\G$ on $M$. 
\end{definition}
Embolic volume is positive for all compact $n$-manifolds $M$ ($n \geqslant 2$), see \cite[Section 7.2.4.]{berger_riemannian_geometry}. Note that on a Riemannian manifold $(M, \G)$, $\sys \pi_1(M, G) \geqslant 2 \inj(M, \G)$. Hence we always have $\emb(M) \geqslant \SR(M)$. Then for closed manifolds with nonzero simplicial volume, Theorem ~\ref{thm_Gromov_SR} implies that the embolic volume is related to simplicial volume. This result is proved by Katz and Sabourau \cite{katz_sabourau_entropy_systolically_05} by a different approach. Moreover, a direct implication of Theorem \ref{thm_02} is that any closed hyperbolic $n$-manifold with $n \geqslant 4$ has
\begin{equation} \label{emb_tri}
 \emb(M) \geqslant D_n \ts \frac{\sigma(M)}{\log^n{(\sigma(M))}},
\end{equation} 
where $D_n$ is the positive constant in Theorem~\ref{thm_02} and only depending on $n$. In this note, we also show that embolic volume of any closed $n$-manifold is related to its triangulation complexity. Compared with Katz and Sabourau's above result, our result is more general, since it includes all closed manifolds with zero simplicial volume.
\begin{theorem} \label{thm_03}
 Let $M$ be a closed $n$-dimensional manifold. Then there exists a positive constant $E_n$ only depending on $n$, such that
 \begin{equation}
  \emb(M) \geqslant E_n \ts \sqrt{\sigma (M)}.   
 \end{equation} 
\end{theorem}

\subsection*{Organization} 
This short note is organized as follows: In Section 2, we discuss J{\o}rgensen and Thurston's theorem of hyperbolic 3-manifolds. Proof of Theorem ~\ref{thm_main} is given in this section. The case of hyperbolic manifolds of high dimensions is contained in Section 3. Then Theorem~\ref{thm_02} is proved. Section 4 concerns embolic volume and triangulation complexity. Proof of Theorem \ref{thm_03} is included in this section.

\subsection*{Acknowledgement} 
This work is supported by the Young Scientists Fund of NSFC (award No. 11901261). Part of this work was done when the author was visiting Tianyuan Mathematical Center in Southwest China. The author wishes to thank Professor Bohui Chen for his invitation and support.

\section{Triangulation and volume of hyperbolic 3-manifolds}

We prove Theorem \ref{thm_main} in this section. The proof is based on J{\o}rgensen and Thurston's theorem of hyperbolic 3-manifolds. J{\o}rgensen and Thurston's work (\cite{thurston_notes_three-manifolds}, also see \cite{kobayashi_rieck_hyperbolic_volume_2011}) implies that triangualtion of a hyperbolic $3$-manifold is related to its volume. 
\begin{theorem}[J{\o}rgensen, Thurston] \label{JT_hyp}
 Let $M$ be a closed hyperbolic $3$-manifold, and $a_0$ be a positive constant.  
 Assume that $\inj(M, \hyp) \geqslant a_0$. Then there exists a triangulation of $M$, with the number $t$ of tetrahedra satisfying
 \begin{equation*}
  t \leqslant K \ts \vol_{\hyp} (M) , 
 \end{equation*}
 where $K$ is a positive constant only depending on $a_0$. 
\end{theorem}
\begin{proof}
For completeness, we show an outline of the proof. More details can be found in \cite[Chapter 5]{thurston_notes_three-manifolds}, \cite{kobayashi_rieck_hyperbolic_volume_2011} and also \cite[Section 4.2]{maria_purcell_treewidth_hyperbolic_volume}.

Let $R = \frac{1}{2} \inj(M, \hyp)$. Assume that $X$ is a maximal subset of points in $M$, so that any two of them having distance at least $R$. The set $X$ is maximal under inclusion. Voronoi cell associated to $x_0 \in X$ is defined to be the the subset
 \begin{equation*}
  V(x_0) = \{ y \in M | dist(y, x_0) \leqslant dist(y, x), \, \text{for any} \, x \in X \, \text{and} \, x \neq x_0 \}. 
 \end{equation*}
Since the ball $B(x_0, \frac{R}{2}) \subset V(x_0)$, the total number of Voronoi cells is at most $\frac{\vol (M)}{\vol (B(x_0, \frac{R}{2}))}$, and thus bounded from above by $C(a_0) \vol (M)$, with $C(a_0)$ a constant only depending on $a_0$. After triangulating each Voronoi cell, we get a triangulation of $M$ with the number $t$ of tetrahedra bounded from above by $K \cdot \vol(M)$, where the constant $K$ is a multiplicative number of $C(a_0)$.
\end{proof}

Sabourau proved that there are only finitely many hyperbolic $3$-manifolds with bounded systolic volume. 
\begin{theorem}[see {\cite[Theorem B]{sabourau_systolic_volume_connected_sums}}] \label{Sabourau_07}
 For a sufficiently large positive number $s_0$, there are only finitely many hyperbolic $n$-manifolds ($n \geqslant 3$) $M$ with $\SR(M) \leqslant s_0$. 
\end{theorem}
\begin{remark}
When $n \geqslant 4$, Theorem \ref{Sabourau_07} is a direct implication of Wang's finiteness theorem and Theorem \ref{thm_Gromov_SR}. Sabourau showed the case of $n = 3$ with a different approach. When $n = 2$, the theorem is not true, since on each closed orientable surface of genus $g \geqslant 2$, there exist infinitely many pairwise non-isometric hyperbolic metrics.   
\end{remark}

\noindent{\bf Proof of Theorem ~\ref{thm_main}:} \\
Let $M$ be a closed hyperbolic $3$-manifold with $\SR(M) \leqslant s_0$, where $s_0$ is a sufficiently large positive constant. According to Sabourau's theorem (see Theorem~\ref{Sabourau_07}), there are only finitely many closed hyperbolic $3$-manifolds $M$ with $\SR(M) \leqslant s_0$. Then we know that the injectivity radii of all these hyperbolic $3$-manifolds $M$ have a common lower bound. Denote this lower bound by $\delta_0$. The constant $\delta_0$ is only depending on $s_0$. Then Theorem~\ref{JT_hyp} implies there exists a positive constant $K(\delta_0)$ such that $M$ admits a triangulation with the number $t$ of tetrahedra satisfying
\begin{equation*}
 t \leqslant K(\delta_0) \ts \vol_{\hyp}(M).
\end{equation*}

The triangulation complexity $c(M)$ of a closed hyperbolic $3$-manifold $M$ satisfies $c(M) \leqslant t$. Hence we have
\begin{align*}
 c(M) & \leqslant t \\  
         & \leqslant K(\delta_0) \ts \vol_{\hyp}(M) \\
         & \leqslant K(\delta_0) \ts \nu_3 \| M \| \\
         & \leqslant K(\delta_0) \ts \nu_3 C_3 \SR(M) \ts \log^3{\left( C_3^{\prime} \ts \SR(M)  \right)} , 
\end{align*}
where $\nu_3$, $C_3$ and $C_3^{\prime}$ are all fixed positive constants. Therefore,
\begin{equation*}
 \SR(M) \geqslant C(s_0) \ts \frac{c(M)}{\log^3{c(M)}} , 
\end{equation*}
where $C(s_0)$ is a positive constant only depending on $s_0$. 
\hfill $ \square $

\section{Hyperbolic manifolds in higher dimensions}

When $n \geqslant 4$, hyperbolic manifolds of dimension $n$ have more rigidity. There are only finitely many hyperbolic $n$-manifolds ($n \geqslant 4$) with volume bounded from above (Wang's finiteness theorem), but this is not true in $n = 3$. Moreover, for hyperbolic $n$-manifolds with $n \geqslant 4$, if hyperbolic volumes are bounded from above, the injectivity radii have a common lower bound. We generalize J{\o}rgensen and Thurston's theorem to hyperbolic manifolds of dimension at least four. Then we prove Theorem~\ref{thm_02} in this section.

\begin{proposition} \label{hyperbolic_n}
 Let $M$ be a closed hyperbolic manifold of dimension $n$, with $n \geqslant 4$. There exists a positive constant $K_n$ depending only on $n$, such that the manifold $M$ admits a triangulation with the number $t$ of $n$-simplices is bounded from above by its volume as follows,
 \begin{equation} \label{tri_n}
  t \leqslant K_n \ts \vol_{\hyp}(M).
 \end{equation}
\end{proposition}
\begin{proof}
Let $(M, \hyp)$ be a closed hyperbolic manifold of dimension $n \geqslant 4$. Set $R = \inj(M, \hyp)$. Suppose that $X \subset M$ is a maximal set of points with any two of them having distance at least $R$. We consider Voronoi cells corresponding to points in $X$. Recall that a Voronoi cell associated to $x \in X$ is
\begin{equation*}
 V(x) = \{ y \in M | dist(x, y) \leqslant dist(x^{\prime}, y), \text{for any} \, x^{\prime} \in X \, \text{and} \, x^{\prime} \neq x \}. 
\end{equation*} 
For any two distinct points $p, q \in X$, $B(p, \frac{R}{2}) \cap B(q, \frac{R}{2}) = \emptyset $ since $dist(p, q) \geqslant R$. The total number of Voronoi cells is thus bounded from by above by
 \begin{align*}
  \frac{\vol_{\hyp}(M)}{\vol_{\hyp}(B(x, \frac{R}{2}))} &  = \frac{1}{\vol_{\hyp}(B(x, \frac{R}{2}))} \ts \vol_{\hyp} (M) \\
   & \leqslant \frac{1}{c_1 (n) e^{(n-1)R/2}} \ts \vol_{\hyp}(M) \\
   & \leqslant \frac{1}{c_1 (n)} \ts \vol_{\hyp}(M), 
 \end{align*}
 where $c_1 (n)$ is a positive constant depending only on $n$. Note that it is proved in \cite{belolipetsky_thomson_systoles_hyperbolic} there exist hyperbolic $n$-manifolds $M$ with arbitrarily small injectivity radius $R$. Hence the above constant $\frac{1}{c_1(n)}$ cannot be further improved.

A triangulation of the hyperbolic manifold $M$ is obtained by triangulating each Voronoi cell $V(x)$ into $n$-simplices. Since two Voronoi cells $V(x)$ and $V(y)$ have a common face if and only if $B(y, R) \subset B(x, \frac{5}{2} R)$, the number of faces of any Voronoi cell $V(x)$ has upper bound of $\frac{\vol_{\hyp}(B(x, \frac{5}{2}R))}{\vol_{\hyp}(B(x, \frac{R}{2}))}$. Hence, the total number of faces for all Voronoi cells is finite and only depending on $n$. In the triangulation, maximal number of $n$-simplices in each Voronoi cell is uniformly bounded from above by the number of its faces. Therefore, there exists a positive constant $K_n$ only depending on $n$, such that the total number of $n$-simplices in the triangulation is bounded from above by $ K_n \ts \vol_{\hyp} (M)$. 
\end{proof}

Now we apply Proposition~\ref{hyperbolic_n} to Proof Theorem~\ref{thm_02}. \\
\noindent{\bf Proof of Theorem \ref{thm_02}:} 
Theorem \ref{thm_02} is proved by using Proposition \ref{hyperbolic_n} and Gromov's theorem (Theorem \ref{thm_Gromov_SR}). Let $M$ be a closed hyperbolic manifold with dimension at least $4$, and $t$ be the number of $n$-simplices of triangulation obtained from Proposition \ref{hyperbolic_n}. Then the triangulation complexity $\sigma(M)$ is bounded from above as follows, 
 \begin{align*}
  \sigma(M) & \leqslant t \\
   & \leqslant K_n \ts \vol_{\hyp} (M) \\ 
   & = K_n \ts \nu_n \| M \| \\
   & \leqslant K_n \nu_n \ts C_n \SR(M) \log^n{\left( C_n^{\prime} \SR(M) \right)} .
 \end{align*}
 Hence, 
 \begin{align*}
  \SR(M) \geqslant D_n \ts \frac{\sigma(M)}{\log^n{\sigma(M) } } ,
 \end{align*} 
 where $D_n$ is a positive constant only depending on $n$. 
$\hfill \square$

\section{Embolic volume of compact manifolds}

We generalize the above results for systolic volume to embolic volume. The proof of Theorem~\ref{thm_03} is given in this section.

Embolic volume is defined in Definition \ref{def_emb}. Berger's embolic inequality (\cite{berger_embolic} or {\cite[Section 7.2.4.]{berger_riemannian_geometry}}) states that for any Riemannian metric $\G$ defined on a compact manifold of dimension $n$, 
\[ \inj(M, \G)^n \leqslant C \ts \vol_{\G}(M) \]
holds, where $C$ is a positive constant. Hence the embolic volume of compact $n$-manifolds is always positive. The relation between embolic volume and other topological invariants is described in \cite[Section 11.2.3.]{berger_riemannian_geometry}. On any closed manifold $M$, we always have $\emb(M) \geqslant \SR(M)$. Hence Gromov's theorem (Theorem \ref{thm_Gromov_SR}) includes a relation between embolic volume and simplicial volume for closed manifolds with nonzero simplicial volume. Moreover, Katz and Sabourau \cite{katz_sabourau_entropy_systolically_05} used a different method to show this result. Then the estimates in Theorem ~\ref{thm_main} and ~\ref{thm_02} also hold if the systolic volume is replaced by embolic volume.

Let $E$ be a sufficiently large positive constant. In \cite{yamaguchi_homotopy_type_finiteness, grove_petersen_wu_geometric_finiteness, grove_petersen_wu_geometric_finiteness_erratum} the finiteness theorems are proved for compact $n$-manifolds $M$ with $\emb(M) \leqslant E$. Therefore, from this point of view, the embolic volume of compact manifolds functions like volume of hyperbolic manifolds. Our theorem in this section provides more evidence to this viewpoint.

The following local embolic inequality of Croke will be used in the proof of Theorem \ref{thm_03}.
\begin{lemma}[see Croke \cite{croke_isoperimetric_80}] \label{croke}
For a Riemannian metric $\G$ defined on compact $n$-dimensional manifold $M$, any metric ball $B(p, r)$ with center $p$ and radius $r \leqslant \frac{1}{2}\inj(M, \G)$ satisfies
\begin{equation} \label{emb_loc}
  \vol(B(x, r)) \geqslant \alpha_n r^n ,
\end{equation} 
 where $\alpha_n$ is a positive constant only depending on $n$. 
\end{lemma}

\noindent{\bf Proof of Theorem \ref{thm_03}:}
For a Riemannian metric $\G$ defined on $M$, let $\inj(M, \G)$ be injectivity radius. The distance function induced by $\G$ is denoted $dist_{\G}( , )$. Assume that $R = \frac{1}{5} \inj(M, \G)$. We say a subset $A \subset M$ is $R$-separated if $dist_{\G}(x, y) \geqslant R$ for all distinct $x, y \in A$. Now let $S$ be a maximal $R$-separated subset of $M$. Here the maximal is under inclusion relation. For $x_0 \in S$, denote by $V(x_0)$ the following Voronoi cell,
 \begin{equation*}
  V(x_0) = \{ y \in S | dist(y, x_0) \leqslant dist(y, x), \forall x \in S \, \text{and} \, x \neq x_0  \} . 
 \end{equation*}
 We have $ B \left(x_0, \frac{R}{2} \right) \subset V(x_0) $. Hence the number of Voronoi cells $V(x_0)$ in $M$ is bounded from above by 
 \[ \frac{\vol_{\G}(M)}{\vol_{\G} (B(x_0, \frac{R}{2})) } .  \]

We obtain a triangulation of $M$ by triangulating each Voronoi cell $V(x)$ into $n$-simplices. In order to let the triangulations on each face of $V(x)$ match, we choose the triangulation of $M$ which induces a triangulation on each face symmetric with respect to combinatorial isomorphisms of the face. The combinatorial types of Voronoi cell $V(x)$ are determined by $R$ and $n$. Hence there are only finitely many such combinatorial types. We let $T$ be the maximal number of simplices in all Voronoi cells $V(x)$. Then the finiteness of combinatorial types of $V(x)$ implies that $T$ is a constant only depending on $R$ and $n$. In fact, two Voronoi cells $V(x)$ and $V(y)$ are adjacent, if and only if $B(x, \frac{R}{2}) \subset B(y, \frac{5R}{2})$. Therefore, the number $T$ has upper bound $\frac{\vol(B(y, \frac{5R}{2}))}{\inf_{x} \vol(B(x, \frac{R}{2}))}$, which is bounded from above by $C_n \emb(M)$ according to Lemma \ref{croke}, where $C_n$ is a positive constant only depending on $n$, and the infimum is taken over all $x$ such that $B(x, \frac{R}{2}) \subset B(x, \frac{5}{2}R)$ holds. The number of $n$-simplices in the triangulation is equal to   
\begin{equation*}
 \frac{\vol_{\G}(M)}{\vol_{\G} (B(x_0, \frac{R}{2})) } \ts T . 
\end{equation*}
By using Croke's local embolic inequality (\ref{emb_loc}) again, we have
\begin{align*}
 \frac{\vol_{\G}(M)}{\vol_{\G} (B(x_0, \frac{R}{2})) } \ts T & \leqslant  \frac{2^n}{\alpha_n^n} \frac{\vol_{\G}(M)}{R^n} \ts T \\
 & \leqslant \beta_n \emb(M, \G)^2 , 
\end{align*}
where $\beta_n$ is a positive constant only depending on $n$. 
Therefore we find a triangulation on $M$ with the number of $n$-simplices bounded from above by $\beta_n \emb(M)^2$. 
$ \hfill \square $

\bibliographystyle{amsalpha}
\bibliography{complexity}
 
\end{document}